\documentclass[11pt,a4paper,reqno, titlepage]{amsart}

 \usepackage{amsaddr}

\usepackage{amsmath,amssymb,amsfonts,amscd,bbm}
\usepackage[mathscr]{euscript}
\usepackage[all]{xy}
\usepackage{color}

\theoremstyle{plain}
\newtheorem{theorem}{Theorem}[section]
\newtheorem{lemma}[theorem]{Lemma}

\newtheorem{corollary}[theorem]{Corollary}
\theoremstyle{definition}
\newtheorem{definition}[theorem]{Definition}
\newtheorem{sdefinitions}[theorem]{Some definitions}
\newcommand{\la}{\langle}
\newcommand{\ra}{\rangle}

\newcommand{\ssquare}{\,\mbox{\scriptsize $\square$}\,}

\newtheorem{remark}[theorem]{Remark}
\newtheorem{remarks}[theorem]{Remarks}
\def\<#1>{\langle\, #1\,\rangle}
\newcommand\norm[1]{\Vert #1 \Vert}

\newcommand{\mA}{\mathscr{A}}

\newcommand{\mV}{\mathfrak{V}}

\newcommand{\bd}{{\ast\ast}}

\newcommand{\F}{\mathbb{F}}

\newcommand{\C}{\mathbb{C}}
\newcommand{\Z}{\mathbb{Z}}
\newcommand{\N}{\mathbb{N}}

\DeclareMathOperator{\supp}{supp}

\newcommand{\WAP}[1][G]{\ensuremath{\mathscr{WAP}}(#1)}
\newcommand{\W}{\mathscr{WAP}}
\numberwithin{equation}{section}

\newcommand{\WAPw}[1][G,w^{-1}]{\ensuremath{\mathscr{WAP}}(#1)}
\newcommand{\UCw}[1][G,w^{-1}]{\ensuremath{\mathscr{UC}}(#1)}

\newcommand{\CBw}[1][G,w^{-1}]{\ensuremath{\mathscr{CB}}(#1)}

\usepackage{amsmath,amscd}

\renewcommand{\emptyset}{\varnothing}

\font\seis=cmr6
\def\CB{\mathscr{CB}}

\def\wap{{\seis{\mathscr{WAP}}}}

\newcommand\restr[2]{{
  \left.\kern-\nulldelimiterspace 
  #1 
  \vphantom{\big|} 
  \right|_{#2} 
  }}

\begin{document}
\newcommand{\I}{\mathbb{I}}

\title[enArity of preduals of von Neumann algebras]{Orthogonal $\ell_1$-sets and extreme non-Arens regularity of  preduals of von Neumann algebras}

\author[Filali and Galindo]{M. Filali \and  J. Galindo}

\thanks{ Research of  the second named  author  supported by Universitat Jaume I, grant UJI-B2019-08}
 \keywords{Arens product, Arens-regular algebra, Banach algebra,  predual, von Neumann algebra, orthogonal set, support projection, extremely non-Arens regular, weighted group algebras, Fourier algebra}

\address{\noindent 
Department of Mathematical Sciences\\University of Oulu, Oulu,
Finland.\\ E-mail: {\tt mfilali@cc.oulu.fi}}
\address{\noindent  Instituto Universitario de Matem\'aticas y
Aplicaciones (IMAC)\\ Universidad Jaume I, E-12071, Cas\-tell\'on,
Spain.\\  E-mail: {\tt jgalindo@uji.es}}

\subjclass[2010]{Primary 47D35; Secondary 22D15;43A46, 43A60,47C15}

\date{\today}

\begin{abstract}
We propose a new
definition for a  Banach algebra $\mA$ to be extremely non-Arens regular, namely that the quotient  $\mA^\ast/\W(\mA)$ of  $\mA^\ast$ with the space of its  weakly almost periodic elements  contains an isomorphic copy of $\mA^\ast.$ This definition is simpler and formally stronger than the original one introduced by Granirer in the nineties.

We then identify sufficient conditions for the predual $\mV_\ast$ of a von Neumann algebra $\mV$ to be
  extremely non-Arens regular
  in this new sense.  These conditions are obtained with the help of orthogonal $\ell_1$-sets of $\mV_\ast.$

	We show that some of the main algebras
   in harmonic analysis satisfy these conditions. Among them,  there is
\begin{enumerate}
\item the weighted semigroup algebra of any  weakly cancellative discrete semigroup, for any  diagonally bounded weight,
\item the weighted group algebra of any  non-discrete locally compact infinite group and for any weight,
\item the weighted measure algebra of any  locally compact infinite group, for any  diagonally bounded weight,
\item the Fourier algebra of any locally compact infinite group having its local weight greater or equal than its compact covering number,
\item the Fourier algebra of any countable discrete  group containing an infinite amenable subgroup.
\end{enumerate}
  \end{abstract}
\maketitle

\section{Introduction}
There are two natural actions of a Banach algebra $\mA$  on its Banach dual $\mA^\ast$
\begin{align}\label{actions}
  \langle  fa,b\rangle=\langle f,ab\rangle\quad\text{and}\quad   \langle
  a.f,b\rangle=\langle f,ba\rangle\end{align}
   for every $f\in \mA^*,$ and $a,b \in \mA.$
   Each of these two actions induces a natural actions of $\mA^\bd$ on $\mA^\ast$
   \begin{align}\label{actionsd}
  \langle \phi \ssquare f,a \rangle=\langle \phi ,fa \rangle\quad \text{and}\quad   \langle
  f\lozenge \phi ,a\rangle=\langle \phi ,a.f\rangle \end{align}
  for every $\phi\in \mA^\bd$, every $f \in \mA^\ast$  and every $a\in \mA$.
In \cite{arens}, Richard Arens used these actions to  extend the product of the  Banach algebra $\mA $ to two Banach algebra products on the second dual $\mathcal{A}^{**}$. These products are naturally defined by
\begin{align*}
  \langle \phi\ssquare \psi,f \rangle=\langle \phi , \psi\ssquare  f \rangle\quad \text{and}\quad   \langle
  \phi\lozenge \psi  ,f\rangle=\langle \psi ,f\lozenge \phi \rangle\end{align*}
  for every $\phi,\psi \in \mA^\bd$ and every $f \in \mA^\ast$. As their definition suggests, these products may differ.

	In fact,  left $\ssquare$-translations,  $\phi\mapsto \phi\ssquare \psi$ are easily seen to be  weak$^\ast$-weak$^\ast$-continuous but
	    right $\ssquare$-translations $\phi\mapsto \psi\ssquare \phi$ may fail to be so when  $\psi\notin\mA$. The same properties hold with right and left $\lozenge$-translations.

Let $\mA_1 $ denote the  closed unit ball of $\mA$. The elements $f\in\mA^*$ having their left orbits $f\mA_1$  relatively weakly compact (equivalently, having their right orbits $\mA_1.f$  relatively weakly compact)
 are called weakly almost periodic and form a closed subspace $\W(\mA)$ of $\mA^*$.
 Due to Grothedieck's double limit criterion, $\W(\mA)$ is precisely the subspace of $\mA^*$ on which the two Arens-products agree, see
\cite{pym}.

     So,  when $\mA^* = \W(\mA)$, i.e., when the quotient
$\mA^*/\W(\mA)$ is trivial, there is only one Arens product,  which is separately weak*-weak*-continuous.
  In such a situation, the algebra $\mA$ is said to be \emph{Arens regular}.
	Otherwise,  the algebra $\mA$ is said to be \emph{non-Arens regular} or \emph{Arens irregular}.
   C$^*$-algebras constitute the paradigmatic example of Arens regular Banach algebras. If $\mA$ is a $C^\ast$-algebra, its universal representation identifies $\mA$ with a norm-closed  algebra of operators on a Hilbert space. By the Sherman-Takeda theorem,   $\mA^\bd$ may then be identified with the closure of the universal representation of $\mA$ in the weak operator topology, and both Arens  products coincide with the multiplication of operators, see \cite{BD}.

Most of the algebras of functions arising in harmonic analysis
 turned out to be   non-Arens regular, even dramatically so. Not only $\W(\mA)$  is often different from  $\mA^\ast$ but  the quotient $\mA^\ast/\W(\mA)$ tends to be  as large as $\mA^\ast$. A name for this situation was coined  by Granirer in  \cite{granirer} when he  called a Banach algebra  $\mA$ \emph{extremely non-Arens regular} (enAr, for short) when $\mA^\ast/\W(\mA)$ contains a closed subspace that has $\mA^\ast$ as a quotient. The group algebra $L^1(G)$ of an infinite locally compact group $G$ is an important example of a Banach algebra that is enAr, see \cite{FG}.
 Extreme non-Arens regularity of the Fourier algebra $A(G)$ is more subtle. First, the question of whether $A(G)$ is non-Arens regular
is still not completely settled. It is known that $A(G)$ is not Arens regular  if $G$ contains an infinite amenable subgroup or if $G$ is not discrete. The first assertion was obtained by Forrest  in \cite[Proposition 3.7]{forrest91}, improving upon results of Lau and Wong \cite[Proposition 5.3]{lau-wong}, and the second  assertion was proved by Forrest
\cite[Corollary 3.2]{forrest91}.
If  $G$ is far enough from being discrete, then $A(G)$ is even enAr. For this to happen, it is enough that the minimal cardinal of an open base at the identity,  $\chi(G)$, is larger than $\kappa(G)$,  the minimal number of compact sets required to cover $G$. This was proved by Granirer \cite{granirer} when $\chi(G)=\omega$ and by Hu \cite{Hu97} in the general case.

For more details, the reader is directed to \cite{dales},
\cite{dales-lau}, \cite{filali-singh}. On the most recent study of the size of quotients of function spaces on a locally compact group, the reader is directed to
 \cite{FG1}.

  Our recent paper \cite{FGnew} devised a general method for proving that a Banach algebra is enAr. This method showed how the two main non-Arens regularity triggers in $L^1(G)$ and $A(G)$ -non-compactness and bounded approximate identity in $L^1(G)$, non-discreteness and  amenability in $A(G)$- are actually particular cases of a single one, the existence of $\ell^1$-bases with a certain multiplicative triangle-like structure in a bounded subset of   the algebra.

When the Banach algebra is the predual of a von Neumann algebra, these $\ell^1$-bases can be taken orthogonal  and then we are able  to construct bounded linear isomorphisms  of $\mA^*$  into
 the  quotient $\mA^\ast/\wap(\mA)$. These isomorphisms are even isometries  when the triangles lie in the unit sphere of $\mA$.
This construction has led us  to propose  a new version of extreme non-Arens regularity by requiring an isomorphic (or an isometric) copy of $\mA^*$  into
 the  quotient $\mA^\ast/\wap(\mA)$. This definition seems to be more natural and still holds in most of the known classes of enAr algebras.

 There is no reason for Banach algebras that happen to be  the  predual of  von Neumann algebras to be enAr. They can be even Arens regular as it the case with the semigroup algebra $\ell^1$ with pointwise product as already observed by Arens in \cite{arens}, or  with many weighted semigroup algebras $\ell^1(S,w)$, see Remark \ref{wAr}.

  We should also mention that the utmost failure of Arens regularity can also be characterized through topological centers. The left and right topological centers of a Banach algebra $\mA$ are, respectively, the sets  $Z_t^{(l)}(\mA)$ and $Z_t^{(r)}(\mA)$  made respectively of elements $\psi\in \mA^\bd$  such that $\phi  \mapsto \psi\ssquare \phi$ and $\phi \mapsto \psi \lozenge \phi$ are  weak*-weak*-continuous.   Dales and Lau \cite{dales-lau} called $\mA$ strongly Arens irregular (sAir for short) when both centers are as small as possible, i.e., when $Z_t^{(l)}(\mA)=Z_t^{(r)}(\mA)=\mA$.
	This happens for a number of Banach algebras in harmonic analysis defined on an infinite locally compact group. We direct the reader to \cite{DLS}, \cite{FiSa07}, \cite{LL1}, \cite{N} and  for the group algebra,
	to \cite{dales-lau}, \cite{FiSa} and \cite{N08}  for the weighted group algebra of an infinite group with a diagonally bounded weight, and to \cite{Loetal} for the measure algebra of an infinite group.
					We may also direct the reader to \cite{FGnotyet, MMM, LL2,LL3}, where various locally compact groups are found for which $A(G)$ is sAir.

It may be worthwhile to note that,  by a recent result of  Losert \cite{lose17}, the Fourier algebra on the free group on two letters $A(\mathbb{F}_2)$ is an example of an algebra which is enAr but not sAir,   \cite{FGnew}. Examples of Banach algebras that are sAir but not enAr can be found in \cite{HN2}.

\section{Outline of the paper}
As stated above,
Granirer defined a Banach algebra $\mA$ to be  enAr when the quotient
$\mA^*/\W(\mA)$ contains a closed linear subspace
which has $\mA^*$ as a continuous linear image.
We start by modifying this definition.

\begin{definition} \label{new} We say that a Banach algebra $\mA$ is \emph{$r$-enAr}, where $r\geq 1$, when there is a linear isomorphism of  $\mA^*$ in the quotient
$\mA^*/\W(\mA)$ with distortion $r$, i.e., when there is a linear isomorphism $\mathcal{E}\colon \mA^\ast \to \mA^*/\W(\mA)$   with
\[\norm{\mathcal{E}}\norm{ \mathcal{E}^{-1}}=r.\]
When $r=1$, the map $\norm{ \mathcal{E}^{-1}} \mathcal{E}$ is a linear isometry and  we say that \emph{$\mA$ is isometrically enAr}.
\end{definition}

Observe that  an  $r$-enAr Banach algebra is always enAr in the sense of Granirer.

Even if we have no example of $\mA$ which is enAr in the sense of Granirer but is not $r$--enAr for any (or some) $r$,  this new definition,  formally stronger, is more natural and easy to handle and is still satisfied by many of the known examples of enAr algebras as we shall see in the present paper.
The authors of this paper already came across with   this concept in a previous paper  (without being formally defined then). It is proved  in \cite[Theorems B and 6.4]{FG} that the group algebra $L^1(G)$ is isometrically enAr for every infinite locally compact group.

In the sections below, we start by reconsidering the main theorem of  \cite{FGnew}. The conditions of that theorem, when combined with the concept of orthogonality available in preduals of von Neumann algebras, are used in Theorem \ref{TTr} to set the
 natural conditions under which these algebras are $r$--enAr for some $r\geq 1$.

This is followed by two sections applying Theorem \ref{TTr} to the Banach algebras in harmonic analysis
which are non-Arens regular.
We extend the main theorem  in \cite{FGnew} to the weighted group algebra and  prove that  the weighted group algebra $L^1(G,w)$ is enAr for any non-discrete locally compact group $G$ and for any weight function on $G.$
If the weight function $w$ is diagonally bounded, then the same is also true for the weighted semigroup algebra $\ell^1(S,w)$ for infinite, discrete, weakly cancellative semigroups $S$, in particular for infinite discrete groups.
We also show that when the weight function $w$ is diagonally bounded on $G$, then the  weighted measure algebra $M(G,w)$ is $r$-enAr for any  infinite locally compact group $G.$

In our last section we show that the conditions of Theorem \ref{TTr} are met when $A(G)$ contains either a  TI-net (a net converging to a topologically invariant mean in $A(G)^{\ast\ast}$) or a   bai-sequence (a sequence converging to a right identity in $A(G)^{\ast\ast}$, see section 3.2 for  the definitions). This implies that $A(G)$ is  isometrically enAr when $G$  satisfies   the condition $\chi(G)\ge \kappa(G)$  or when $G$ is a second countable group    containing a non-compact amenable open subgroup. This strengthens the corresponding results in
\cite{granirer}, \cite{Hu97} and \cite{FGnew}.

\section{Triangles and weakly almost periodic functionals}
By a directed set, it is always meant a set $\Lambda$ together  with a preorder $\preceq$ with the additional property that every pair of elements has an upper bound. We will use a single letter,  $\Lambda$ usually, to  denote a  directed set, the existence of $\preceq$ is implicitly assumed.

Definitions \ref{def:1} through \ref{def:last} were introduced in \cite{FGnew}.

 \begin{definition}\label{def:1}
   Let $\Lambda$ be a directed set and for every $\xi \in \Lambda$, define $\xi^+=\left\{\alpha \in \Lambda\colon \xi \prec \alpha\right\}$.
   We define the \emph{true cardinality} of $\Lambda$ as $\mathrm{tr}\left|\Lambda\right|=\min_{\xi\in \Lambda} \left|\xi^+\right|$.
 \end{definition}

 \begin{lemma}[van Douwen, Lemma of \cite{vanD}]\label{vanD}
 If $\Lambda$ is a directed set, then $\Lambda$ admits a pairwise disjoint collection of $\mathrm{tr}\left|\Lambda\right|$-many cofinal subsets of $\Lambda$ each having cardinality  $\mathrm{tr}\left|\Lambda\right|$. \end{lemma}

\begin{remark}
In our application to  the algebras in harmonic analysis,  our set $\Lambda$ will  always be the initial ordinal associated to a cardinal $\eta$
(see, e.g., Theorems \ref{thm:quotient-non-separable} and \ref{Gentent}).

Since $|\xi^+|=\eta$ for each $\xi<\eta$, the  true cardinality of $\eta$ is then its cardinality.
\end{remark}

 \begin{definition}
  Let   $(\Lambda,\preceq)$ be a directed set and  $\mathcal U$ a subset of $\Lambda\times \Lambda$. We say that
	\begin{enumerate}
\item  $\mathcal U$ is \emph{vertically cofinal}, when for every $\alpha\in \Lambda$, there exists $\beta(\alpha)\in\Lambda$ such that $(\alpha,\beta)\in \mathcal U$ for every $\beta\succeq \beta(\alpha)$.
\item $\mathcal U$ is \emph{horizontally cofinal}, when for every $\beta\in\Lambda$, there exists $\alpha(\beta)\in \Lambda$ such that $(\alpha,\beta)\in \mathcal U$ for every $\alpha\succeq \alpha(\beta)$.
\end{enumerate}
\end{definition}

\begin{definition}\label{dindex}
  Let $\mathcal{U}$  and $X$ be two  sets. We say that
	\begin{enumerate}
\item  $X$   is \emph{indexed} by $\mathcal U$, when  there exists  a surjective map
$x:\mathcal U\to X$.
 When     $\mathcal U\subset \Lambda\times \Lambda$, for  some other set $\Lambda$,  we say that $X$ is \emph{double-indexed} by
$\mathcal U$ and write
$X=\{x_{\alpha\beta}:(\alpha,\beta)\in \mathcal U\}$, where  $x_{\alpha\beta}=x(\alpha,\beta)$.
\item If $X$ is  double-indexed by $\mathcal U$,  we say it is \emph{vertically injective} if $x_{\alpha\beta}=x_{\alpha^\prime \beta^\prime}$ implies $\beta=\beta^\prime$ for every $(\alpha,\beta)\in\mathcal U$. If $x_{\alpha\beta}=x_{\alpha^\prime \beta^\prime}$ implies $\alpha=\alpha^\prime$ for every $(\alpha,\beta )\in \mathcal U$, we say that $X$ is  \emph{horizontally injective}.
\end{enumerate}
\end{definition}

 \begin{definition}\label{def:last}
  Let $\mA$ be a Banach algebra,  $(\Lambda,\preceq)$ be a directed set
	  and  $A$ and $B$ be two subsets of $\mA$ indexed by $\Lambda$ so that
	$A= \left\{a_\alpha\colon \alpha\in \Lambda\right\}$ and $B=\left\{b_\alpha\colon \alpha\in  \Lambda\right\}$.
 \begin{enumerate}
\item The  sets
 \[T^{u}_{AB}=\left\{a_\alpha  b_\beta\colon \alpha, \beta \in \Lambda,\; \alpha\prec \beta\right\}\quad \mbox{ and } \quad
T^{l}_{AB}=\left\{a_\alpha  b_\beta\colon \alpha, \beta \in \Lambda,\; \beta\prec \alpha\right\}\] are called, respectively, the \emph{upper} and \emph{lower triangles defined by $A$ and $B$}.

\item A set $X\subseteq \mA$
is said to  \emph{approximate  segments in $T_{AB}^{u}$}, if there exists a horizontally cofinal set $\mathcal U$ in $\Lambda\times\Lambda$
so that $X$ is double-indexed as $X=\{x_{\alpha\beta}\colon (\alpha,\beta)\in \mathcal U\} $, and for each $\alpha\in \Lambda$,
 \[\lim_{\beta}\left\|x_{\alpha\beta} - a_\alpha b_\beta\right\|=0.\]

\item A set $X\subseteq \mA$ is said to \emph{ approximate segments in $T_{AB}^{l} $}, if there exists a vertically cofinal set $\mathcal U$ in $\Lambda\times\Lambda$ so that $X$ is double-indexed as $X=\{x_{\alpha\beta}\colon (\alpha,\beta)\in \mathcal U\} $, and for each $\beta\in \Lambda$,
\[\lim_{\alpha}
 \left\|x_{\alpha\beta}-a_\alpha b_\beta\right\|=0.\]
 \end{enumerate}\end{definition}

\begin{definition}\label{def:supp}
Let $\mV$ be a von Neumann algebra with predual  $\mV_\ast$, and let $a\in \mV_\ast$ be a  positive normal functional. The \emph{support projection} of $a$  is the smallest projection  $S(a)\in \mV$ such that $\langle a,S(a)\rangle=\langle a,I\rangle=\|a\|$.
\end{definition}

\begin{definition}\label{def:appl1}
Let $\mV$ be a von Neumann algebra and let   $A\subseteq \mV_\ast^+$. We say that $A$ is an  \emph{orthogonal} $\ell^1(\eta)$-set with bound $M$ and  constant $K>0$  if
\begin{enumerate} \item
 $|A|=\eta$,
\item    $S(a)\,S(a^\prime)=0$ whenever $a,a^\prime \in A$, $a\neq a^\prime$ and
\item  $ K\leq \norm{a}\leq M$ for every $a\in A$.
\end{enumerate}
\end{definition}

\begin{remarks}\label{rems}
\begin{enumerate}

    \item It follows from the Definition of $S(a)$  and the Cauchy-Schwarz inequality,  that for every $a\in \mV_\ast^+$, we have  $S(a)a= a.S(a)=a$.
		Here we are looking at $a$ as an element in $\mV^*$ and using the actions in \ref{actions}.
		\item If  $S(a)$ and $S(b)$ are orthogonal, then $\langle a,S(b)\rangle=0$. This follows from the previous point, making $S(a)a$ act  on $S(b)$.
\item
We have chosen the term \emph{orthogonal $\ell^1(\eta)$-set} because these sets
are   equivalent to the unit $\ell^1$-basis, i.e.,  the closed vector space they span is isomorphic to $\ell^1(\eta)$. To see this  let $a_1,\ldots, a_k\in A$ and $z_1,\ldots, z_k\in \C$.
Taking into account that $\left\|\sum_{k=1}^{n}
\frac{\overline{z_k}}{|z_k|}S(a_k)\right\|_{_{\mV}}\leq 1$, we have
that \begin{align*}
\left\|\sum_{k=1}^{n}z_k a_k\right\|&\geq
\Bigl| \left\langle \sum_{k=1}^{n}z_k a_k, \sum_{k=1}^{n}
\frac{\overline{z_k}}{|z_k|}S(a_k)\right\rangle\Bigr|
\\
&=\left|\sum_{k=1}^n \sum_{j=1}^n z_k \frac{\overline{z_j}}{|z_j|}\left\langle S(a_j),a_k\right\rangle\right|
    \\&= \left|\sum_{k=1}^n  |z_k|\left\langle S(a_k),a_k\right\rangle
\right|=\sum_{k=1}^n |z_k| \|a_k\|\geq K\sum_{k=1}^{n}|z_k|,
\end{align*}
and this shows that $A$ is equivalent to the unit $\ell^1$-basis.
\end{enumerate}\end{remarks}
\subsection{Auxiliary lemmas}
The following definition and its consequence, recorded in \cite{FG}, will prove convenient to exploit Definition \ref{def:appl1}.

 \begin{definition} Let $E_1$ and $E_2$ be Banach spaces, $\mathcal T\colon E_1\to E_2$ be a bounded linear map, $F$ be a closed subspace of $E_2$,
and let $c>0$. We say that $\mathcal T$ is $c$-preserved by   $F$    when the following property holds
\begin{align}\label{preserve}
   \tag{$\ast$}
& \left\|\mathcal T\xi-\phi\right\|\geq c\|\xi\|, \quad \mbox{for all } \phi\in F \mbox{ and }\xi \in E_1.
\end{align}
  \end{definition}

The proof of next lemma is similar to that of \cite[Lemma 2.2]{FG}.

\begin{lemma}\label{quotient}
Let $\mathcal T\colon E_1\to E_2$ be a bounded linear isomorphism of the Banach spaces $E_1$ into $E_2$  and let $D,$ $F$     be closed linear subspaces of $E_2$ with $D\subseteq F$. Denote by  $Q \colon E_2\to  E_2/D$ the quotient map.  If $T$ is $c$-preserved by $F$, then  the  map $Q\circ \mathcal T\colon E_1\to E_2/D$  is a linear  isomorphism with distortion at most $\frac{\|\mathcal T\|}{c}$.

    \end{lemma}
	
\subsection{Weak almost periodicity}
We make concrete here  the facts about weak almost periodicity    mentioned in the introduction.
 \begin{definition}
 Let $\mA$ be a Banach algebra. A functional $f\in \mA^\ast$ is said to be \emph{weakly almost periodic} if the left orbit $\mA_1\cdot f=\left\{a \cdot  f\colon a\in \mA_1\right\} $ is  relatively weakly compact.
 \end{definition}
 \begin{theorem}[\cite{pym}]
Let $\mA$ be a Banach algebra and $f\in \mA^\ast$. The following are equivalent:
\begin{enumerate}
\item $f$ is weakly almost periodic.
\item The  right orbit $f\mA_1=\left\{  fa\colon a\in \mA_1\right\} $ is relatively weakly    compact.
\item (Grothendieck's double limit criterion) \[\lim_\alpha\lim_\beta f(a_\alpha b_\beta)=\lim_\beta\lim_\alpha f(a_\alpha b_\beta)\] for every pair of bounded nets $(a_\alpha)_\alpha$ and $(b_\beta)_\beta$ in $ \mA$ for which both limits exist.
\end{enumerate}
 \end{theorem}

One more lemma is needed before we state our main theorem. The following might be a well-known fact, but we found no reference for it.

\begin{lemma}\label{limit}
Let  $\{P_\xi\}_{\xi\prec\eta}$ be a family of orthogonal projections on a Hilbert space $H$ and ${\bf v}= (z_\xi)_{\xi\prec\eta}\in \ell^\infty(\eta).$
Then $\sum_{\xi\prec\eta}z_\xi P_\xi$ converges strongly to  a bounded operator $P_{\bf v}$ on $H$. Moreover, $\|P_{\bf v}\|=\|\bf v\|_\infty.$
\end{lemma}

\begin{proof} Note first that, for every finite subset $F$ of $\eta$,  $P_F=\sum_{\xi\in F} P_\xi$ is a projection and, thus, $\norm{P_F}=1$.

Pick now $\epsilon>0$ and $\xi_0\prec\eta$ such that $\|{\bf v}\|_\infty -\epsilon<|z_{\xi_0}|$. Then, for
a finite subset $F$ of $\eta$ and a vector $v$ in the Hilbert space $H$, we have

\begin{align*}\| \sum_{\xi\in F}z_\xi P_\xi v\|^2&=\la\sum_{\xi\in F}z_\xi P_\xi v,\sum_{\xi\in F}z_\xi P_\xi v \ra=\sum_{\xi\in F} |z_\xi|^2 \la P_\xi v, P_\xi v\ra \\&< \sum_{\xi\in F} (|z_{\xi_0}| +\epsilon)^2 \la P_\xi v, P_\xi v\ra
 \leq  (|z_{\xi_0}|+\epsilon)^2\|P_F\|\|v\|^2
= (|z_{\xi_0}|+\epsilon)^2\|v\|^2.
\end{align*}
Therefore,  \begin{equation}\label{bound}\| \sum_{\xi\in F}z_\xi P_\xi v\|^2=\sum_{\xi\in F} |z_\xi|^2 \la P_\xi v, P_\xi v\ra=\sum_{\xi\in F} |z_\xi|^2 \la P_\xi v, v\ra
\le  \|{\bf v}\|_\infty^2\|v\|^2.\end{equation}

If we denote the set of all finite subsets of $\eta$ by $\eta^{<\omega}$ and direct it by set inclusion, this shows that the net $\left(\sum_{\xi\in F}  |z_\xi|^2 \la P_\xi v, v\ra\right)_{F\in\eta^{<\omega}}$ is increasing  and bounded for each $\bf v\in \ell^\infty(\eta)$ and $v\in H$.
Since the following identity \[\left\|\sum_{\xi\in F_1}z_\xi P_\xi v -\sum_{\xi\in F_2}z_\xi P_\xi v\right\|^2=\sum_{\xi\in F_1}|z_\xi|^2 \la P_\xi  v,v\ra-\sum_{\xi\in F_2}|z_\xi|^2 \la P_\xi v, v\ra\] holds for every $ F_1, F_2\in \eta^{<\omega}$,
this shows that  $\left(\sum_{\xi\in F}z_\xi P_\xi v\right)_{F\in\eta^{<\omega}}$ is a norm Cauchy net in $H$. If we denote  $P_{\bf v}v$
 in the norm of $H.$ Clearly,   $P_{\bf v}$ is a bounded operator on $H$ and is the limit of $\sum_{\xi<\eta}z_\xi P_\xi$ in the strong operator topology.
The second statement follows from (\ref{bound}).
\end{proof}

 \begin{theorem}\label{TTr}
 Let $\mA$ be a Banach algebra  and suppose that $\mA$ is a subalgebra of the predual $\mV_\ast$ of a  von Neumann algebra $\mV$. Let
   $\eta$  be an infinite  cardinal number and suppose that  $\mA$ contains two bounded
 subsets $A$ and $B$ indexed by a  directed set $(\Lambda,\preceq)$ with $\mathrm{tr}\left|\Lambda\right|=\eta$.
Suppose also that $\mA$ contains two other disjoint  sets $X_1$ and $X_2$ with the following properties
 \begin{enumerate}
   \item  $X_1$  and $X_2$ approximate  segments in $T^{u}_{AB}$ and $T_{AB}^{l} $, respectively.
       \item $X_1\cup X_2$ is an orthogonal $\ell^1(\eta)$-set (as a subset of $\mV_\ast$) with constant $K$ and bound $M$, i.e., with
			$K\le \|x\|\le M$ for every $x\in X_1\cup X_2$.
       \item $X_1$ is vertically injective and $X_2$ is horizontally injective.
			 \end{enumerate}
        Then there is a linear  isomorphism
   $\mathcal{E}\colon \ell^\infty(\eta) \to \mA^\ast/\wap(\mA)$ with distortion at most $\frac{M}K$.
	
	In particular, $\norm{\mathcal E^{-1}}\mathcal E$ is a linear isometry when $K=M.$ 	
	\end{theorem}

 \begin{proof}
 Put $A=\left\{a_\alpha \colon \alpha\in\Lambda \right\}$ and
  $B=\left\{b_\beta\colon \beta\in\Lambda \right\}$. Let \begin{align*}X_1=\{x_{\alpha\beta}\colon (\alpha,\beta)\in \mathcal U_1\}  \quad\text{ and }
	 X_2=\{x_{\alpha\beta}\colon (\alpha,\beta)\in\mathcal U_2\}  \end{align*} be the sets which approximate
  $T_{AB}^u $ and $T_{AB}^l$, respectively, where $\mathcal U_1$ and $\mathcal U_2$ are horizontally and vertically cofinal in $\Lambda\times \Lambda$, respectively.

Introduce on $\Lambda\times \Lambda$ an equivalence relation by the rule $(\alpha,\beta)\sim (\alpha^\prime,\beta^\prime) $ if and only if $x_{\alpha\beta}=x_{\alpha^\prime\beta^\prime}$. Note that, by vertical/horizontal injectivity, $(\alpha,\beta)\sim (\alpha^\prime,\beta^\prime)$ implies $\beta=\beta^\prime$ when $\beta\succ \alpha$  and implies $\alpha=\alpha^\prime$ when $\alpha \succ\beta$.
We shall denote the equivalence class of $(\alpha,\beta)\in\Lambda\times \Lambda$ by $[\alpha,\beta]$.
Restrict this equivalence relation to $\mathcal U_1$ and $\mathcal U_2.$
Denote the quotient spaces  \[\mathcal U_1/\!\sim\quad\text{and}\quad\mathcal U_2/\!\sim,\] respectively, by $R_1$ and $R_2$.
By \cite[Lemma, p. 61]{vanD}, we can partition  the set $\Lambda$  into $\eta$-many cofinal subsets $\Lambda_{\xi}$, with the cardinality of each  $\Lambda_{\xi}$ equals $\eta$.
 Then,
for an element $\mathbf{v}=
    \bigl(z_\xi\bigr)_{\xi\prec\eta}\in  \ell^\infty(\eta)$, we consider the sum

		    \[\sum_{\xi\prec\eta}z_\xi\left(\sum_{\overset{[\alpha,\beta]\in R_1}{\beta\in \Lambda_{\xi}}} S(x_{\alpha\beta})-
\sum_{\overset{[\alpha,\beta]\in R_2}{\alpha\in \Lambda_{\xi}}}S(x_{\alpha\beta})\right) .\]

  Since $X_1\cup X_2$ is an orthogonal set,  and $x_{\alpha\beta}\neq x_{\alpha^\prime\beta^\prime}$ when $[\alpha,\beta] \neq [\alpha^\prime,\beta^\prime ]$, all the projections appearing in this sum are pairwise orthogonal. So by Lemma \ref{limit}, the above sum converges strongly to an element of $\mV$.
		We label this element as $P_{\mathbf{v}}$.

 We now define $\mathcal T\colon \ell^\infty(\Lambda)\to \mA^\ast$ by
  \[\mathcal{T}(\mathbf{v})=
  \restr{P_{\mathbf{v}}}{\mA},\quad\text{for each}\quad\mathbf{v}\in  \ell^\infty(\eta).\]

  Recalling Lemma \ref{bound} it is obvious, that for every $\mathbf{v}\in \ell^\infty(\eta)$, \begin{equation}\label{norme}
\norm{\mathcal{T}(\mathbf{v})}_{_{\mA^\ast}}\leq
   \norm{P_{\mathbf{v}}}_{_{\mV}}=  \norm{\mathbf{v}}_\infty.\end{equation}
It is easy to check in fact that $\mathcal T$ is an isomorphism.
(When $\mA=\mV_*$, $\mathcal T$ is even an isometry.)

Next we prove that $\mathcal{T}$ is  $\frac{K}{M}$-preserved by $\wap(\mA)$. We consider $\mathbf{v}=(z_\xi)_{\xi\prec\eta}\in \ell^\infty(\eta)$ and $\phi\in \wap(\mA)$ and let $\varepsilon>0$ be fixed as well.
Let now $\xi\prec\eta$ be fixed.
Since $A$ and $B$ are bounded and $\phi\in \wap(\mA)$, we can assume, after taking suitable subnets on the  cofinal set $\Lambda_\xi$,  that the following equality holds
  \[ \lim_\alpha\lim_\beta \left\langle \phi,a_\alpha b_\beta\right\rangle=\lim_\beta\lim_\alpha \left\langle \phi,a_\alpha b_\beta\right\rangle.\]

	Mark these iterated limits by $L$,  put $L_\alpha=\lim_\beta \left\langle \phi,a_\alpha b_\beta\right\rangle$ and $M_\beta = \lim_\alpha\left\langle \phi,a_\alpha b_\beta\right\rangle$.
	Then $\lim_\alpha L_\alpha=\lim_\beta M_\beta=L$, and so for a fixed $\varepsilon >0$, we may choose $\alpha_0$ and $\beta_0$ in $\Lambda_{\xi}$   such that
 \begin{equation}\label{1}|L-L_{\alpha_0}|<\varepsilon/4
  \quad\text{and}\quad
  |L-M_{\beta_0}|<\varepsilon/4.
  \end{equation}
  For these fixed $\alpha_0$ and $\beta_0$, there are $\beta_1$ and $\alpha_1$ in $\Lambda_\xi$ such that
  \begin{align}\label{2}
    \left|\phi(a_{\alpha_0 } b_\beta)-L_{\alpha_0}\right|&<\varepsilon/4\quad \mbox{ for all } \beta \succ \beta_1,\\\label{3}
    \left|\phi(a_\alpha b_{\beta_0})-M_{\beta_0}\right|
    &<\varepsilon/4\quad \mbox{ for all } \alpha \succ \alpha_1.
  \end{align}

 Putting together \eqref{1}--\eqref{3}, we obtain  $\alpha_0,\alpha_1,$ $\beta_0,\beta_1\in \Lambda_\xi$,   such  that
  \begin{equation}\label{eqwap} \left| \left\langle \phi,a_{\alpha_0} b_{\beta}\right\rangle- \left\langle \phi, a_{\alpha} b_{\beta_0}\right\rangle\right|<\varepsilon \quad \mbox{ for all } \beta\succ \beta_1 \mbox{ and } \alpha\succ \alpha_1.\end{equation}

 Since \[\lim_\beta \|a_{\alpha_0}b_\beta-x_{\alpha_0\beta}\|_{_{\mA}}=
\lim_\alpha \|a_\alpha b_{\beta_0}-x_{\alpha\beta_0}\|_{_{\mA}}=0 \] and $\Lambda_{\xi}$ is cofinal, we can find  $\alpha_2,
\beta_2\in \Lambda_{\xi}$ with $\alpha_0,\beta_1\prec \beta_2 $ and $\beta_0,\alpha_1\prec \alpha_2 $ such that  \begin{equation*}
\label{eapp}
\left\| a_{\alpha_0}b_{\beta_2}-x_{\alpha_0\beta_2}\right\|_{_{\mA}}
\leq \varepsilon \mbox{ and } \left\| a_{\alpha_2}b_{\beta_0}-x_{\alpha_2\beta_0}\right\|_{_{\mA}}
\leq \varepsilon,
    \end{equation*}
and so \begin{equation}\label{eappp}
\left\| a_{\alpha_0}b_{\beta_2}- a_{\alpha_2}b_{\beta_0}\right\|_{_{\mA}}\le 2M+2\varepsilon.
    \end{equation}

  Recalling  that (see (ii) of Remarks \ref{rems})
	\[\langle S(x_{\alpha\beta}),x_{\alpha'\beta'}\rangle =0\]
	whenever $[\alpha,\beta]\neq [\alpha', \beta']$ and taking into account that  \[\langle S(x_{\alpha\beta}),x_{\alpha\beta}\rangle=
	\|x_{\alpha\beta}\|\ge K\quad\text{ for every}\quad\alpha,\; \beta \in \Lambda,\]
	we see that
  \begin{equation}\label{nofi} \left|\left\langle \mathcal{T}(\mathbf{v}), x_{\alpha_0\beta_2}-x_{\alpha_2\beta_0}\right\rangle\right|=
  |z_\xi|\left( \left\langle S(x_{\alpha_0\beta_2}),x_{\alpha_0\beta_2}\right\rangle+\left\langle S(
  x_{\alpha_2\beta_0}),x_{\alpha_2\beta_0}\right\rangle\right)\ge 2K|z_{\xi}|.\end{equation}

Using \eqref{eapp} and \eqref{nofi}, it follows that

  \begin{align*}
  \hspace{-0.3cm} \Bigl|\left\langle
\mathcal{T}(\mathbf{v})-\phi,    a_{\alpha_0}b_{\beta_2}-a_{\alpha_2} b_{\beta_0}
   \right\rangle\Bigr|
	&=
 \Bigl|\left\langle\mathcal{T}(\mathbf{v}), a_{\alpha_0}b_{\beta_2}-a_{\alpha_2} b_{\beta_0}
   \right\rangle-
 \left\langle
 \phi,    a_{\alpha_0}b_{\beta_2}-a_{\alpha_2} b_{\beta_0}
   \right\rangle\Bigr|&=\\
	\notag &=
 \Bigl|\left\langle\mathcal{T}(\mathbf{v}),
  (a_{\alpha_0}b_{\beta_2}-a_{\alpha_2} b_{\beta_0})-(x_{\alpha_0\beta_2}-x_{\alpha_2\beta_0})\right\rangle\\&+
  \left\langle\mathcal{T}(\mathbf{v}),
 x_{\alpha_0\beta_2}-x_{\alpha_2\beta_0}\right\rangle-\left\langle
 \phi,    a_{\alpha_0}b_{\beta_2}-a_{\alpha_2} b_{\beta_0}
  	\right\rangle\Bigr|
	\\
&\geq
  \Bigl|\left\langle\mathcal{T}(\mathbf{v}),
 x_{\alpha_0\beta_2}-x_{\alpha_2\beta_0}\right\rangle\Bigr|-
 \Bigl|\left\langle\mathcal{T}(\mathbf{v}),   a_{\alpha_0}b_{\beta_2}-x_{\alpha_0\beta_2}\right\rangle\Bigr|\\&-
\Bigl|\left\langle\mathcal{T}(\mathbf{v}),   a_{\alpha_2} b_{\beta_0}-x_{\alpha_2\beta_0}\right\rangle\Bigr|
-\Bigl|\left\langle
 \phi,    a_{\alpha_0}b_{\beta_2}-a_{\alpha_2} b_{\beta_0}
   \right\rangle\Bigr| 	\\
&\geq
 2 K \left|z_\xi\right|- 2\varepsilon\left\|\mathcal{T}(\mathbf{v})
\right\|_{\mA^\ast}-\varepsilon\\&
\geq  2 K \left|z_\xi\right|- 2\varepsilon \left\|\mathbf{v}\right\|_\infty-\varepsilon.
  \end{align*}

Next, we use   \eqref{eappp} to  obtain
  \begin{align}\label{aiend}
  \left\| \mathcal{T}(\mathbf{v})-\phi\right\|_{\mA^\ast}
  &\geq \frac{1}{\norm{a_{\alpha_0}b_{\beta_2}-a_{\alpha_2}b_{\beta_0}}_{\mA}}\left(
   2 K \left|z_\xi\right|- 2\varepsilon\left\|\mathbf{v}\right\|_\infty- \varepsilon \right)\\&\geq \frac{1}{2M+2\varepsilon}\left(
   2 K\left|z_\xi\right|- 2\varepsilon\left\|\mathbf{v}\right\|_\infty- \varepsilon \right).\notag
  \end{align}

Since $\xi\prec\eta$ and $\varepsilon>0$ were chosen arbitrarily, we conclude that, (use \eqref{norme} for the right hand side inequality)
  \begin{align*}
\frac{K}{M} \left\|\mathbf{v}\right\|_\infty\leq   \left\| \mathcal{T}(\mathbf{v})-\phi\right\|_{_{
  \mA^\ast}}
&
 \leq \norm{\mathbf{v}}_\infty \quad\text{for every}\quad \mathbf{v}\in\ell^\infty(\eta),\; \phi\in \wap(\mA).\end{align*}
The map $\mathcal{T}$ is therefore a  linear isomorphism
that is  $\frac{K}{M}$-preserved by $\wap(\mA)$.

	If $Q$ is the quotient map of $\mA^*$ into $\mA^*/\wap(\mA)$, then  Lemma \ref{quotient} shows that  $\mathcal E=Q\circ\mathcal T$ is the sought after  isomorphism with distortion at most $ \frac{M}K$.

If $K=M,$ then $\norm{\mathcal E^{-1}}\mathcal{E}$ becomes clearly an isometry.    \end{proof}

The conditions of Theorem \ref{TTr} imply clearly  that $\mA$ is non-Arens regular. If, in addition $d(\mA)=\eta$,
    then $\mA$ is even enAr.  This follows from the following well-known fact (see, e.g. \cite{Hu97}). Recall that
the density character of a normed space $\mA$, denoted by $d(\mA)$, is the cardinality of the smallest
norm-dense subset of $\mA$.

 \begin{lemma} \label{hul} If $\mA$ is a normed space with density character $d(\mA)=\eta,$
then there is a linear isometry of $\mA^*$ into $\ell^\infty(\eta).$
\end{lemma}

 \begin{proof}If $\{x_\alpha:\alpha<\eta\}$ is  a norm-dense subset in the unit ball of $\mA$, the required isometry $\mathcal I: \mA^*\to   \ell^\infty(\eta)$ is defined by:
 \[ \mathcal I(\psi)=v_\psi,\] where
    $v_\psi$ in $\ell^\infty(\eta)$ is given
 by $v_\psi(\alpha)=\langle\psi, x_\alpha\rangle.$
\end{proof}

 \begin{corollary}\label{enAr:gen}   Let $\mA$ be as in Theorem \ref{TTr} and suppose that
     $d(\mA)=\eta$. Then $\mA$ is $r$-enAr with $r\le \frac{M}K$.
		In particular, $\mA$ is isometrically enAr when $K=M.$
    \end{corollary}

		 The presence of a bai or of a TI-net   is usually the reason behind the non-Arens regularity of a given Banach algebra. A weaker form of these nets is in fact enough to deduce non-Arens regularity. Here are the  necessary definitions.

\begin{definition} In a Banach algebra $\mA$,  a net $\{a_\alpha\colon \alpha\in \Lambda\}$, with $\|a_\alpha\|=1$ for every $\alpha\in\Lambda,$ is a \emph{weak  bounded  approximate identity} (weak bai for short) if  \[\lim_\alpha\norm{ a_\alpha a_\beta-a_\beta}=\lim_\alpha\norm{a_\beta a_\alpha-a_\beta}=0\quad\text{
for each}\quad \beta \in \Lambda.\]
\end{definition}
			
\begin{definition} Let $\mV$ be a von Neuman algebra. A net $\{a_\alpha\}_{\alpha \in \Lambda}$ of normal states of $\mV$ is a
\emph{weak TI-net} if
\[\lim_\alpha\norm{ a_\alpha a_\beta-a_\alpha}=\lim_\alpha\norm{ a_\beta a_\alpha -a_\alpha }=   0\quad\text{for each}\quad \beta\in  \Lambda.\]
 \end{definition}

If we require that $\lim_\alpha\norm{ a_\alpha a -a_\alpha}=\lim_\alpha\norm{ aa_\alpha  -a_\alpha}=0$
 for every normal state $a$ of $\mV$, and not only for members of the net itself,  then we obtain the familiar concept of a \emph{TI-net}.
 		Here    TI stands for topological invariance, the term was introduced by Chou \cite{chou82TIM}. They also appeared in the work by Lau,
		see for example \cite{L}.

		A general Banach algebra satisfying the conditions	of the following theorem was proved to be enAr in the sense of Granirer in \cite[Theorems 4.2 and 4.4]{FGnew}.
		When $\mA$ is in addition a subalgebra of the predual of a von Neumann algebra, Theorem \ref{TTr} implies as we see next that $\mA$
		is isometrically enAr.
		
\begin{theorem} \label{TI} Let $\mA$ be a Banach algebra  and suppose that $\mA$ is a subalgebra of the predual $\mV_\ast$ of a  von Neumann algebra $\mV$. Let $\eta$ be an infinite cardinal number and suppose that
$\mA$ contains either a weak bai or a weak TI-net $\{a_\alpha\}_{\alpha\in \Lambda}$ of true cardinality $\eta$
such that  $\{a_\alpha\colon\alpha\in \Lambda\}$  is an orthogonal   $\ell^1(\eta)$-set. Then
\begin{enumerate}
\item there is an isometry    $\mathcal{E}\colon \ell^\infty(\eta) \to \mA^\ast/\wap(\mA)$,
\item in particular, $\mA$ is non-Arens regular,
\item  $\mA$ is isometrically enAr, if in addition $d(\mA)\le \eta$.
\end{enumerate}
\end{theorem}

			\begin{proof}
Take $\Lambda_1,\: \Lambda_2\subset \Lambda$ with $\Lambda_1\cap \Lambda_2=\emptyset$  in such a way that  both  $\Lambda_1$ and $\Lambda_2$ are cofinal for  $\preceq$.  Put then
              \[A=\{a_{\alpha} \colon \alpha \in \Lambda_1\}\quad\text{ and}\quad B=\{a_{\alpha}\colon \alpha \in \Lambda_2\}. \]

If $\{a_\alpha\}$ is a weak bai, define the elements  $x_{\alpha\beta}$ by
				
										\[x_{\alpha\beta}=\begin{cases} a_\alpha,  \quad\text{ if}\quad  (\alpha,\beta)\in \Lambda\times \Lambda_1,\;\beta\succ \alpha\\ a_\beta,  \quad\text{ if}\quad (\alpha,\beta)\in \Lambda_2\times \Lambda,\; \alpha\succ \beta.\end{cases}\]

If $\{a_\alpha\}$ is a weak TI-net, define the elements  $x_{\alpha\beta}$ by
				
										\[x_{\alpha\beta}=\begin{cases} a_\beta,  \quad\text{ if}\quad  (\alpha,\beta)\in \Lambda\times \Lambda_1,\;\beta\succ\alpha\\ a_\alpha,  \quad\text{ if}\quad (\alpha,\beta)\in \Lambda_2\times \Lambda,\; \alpha\succ\beta.\end{cases}\]
					
						In each case, let

													\[X_1=\{x_{\alpha\beta}\colon (\alpha,\beta)\in \Lambda\times\Lambda_1,\;\beta\succ \alpha\} \;\text{ and}\; X_2=\{x_{\alpha\beta}\colon  (\alpha,\beta)\in \Lambda_2\times\Lambda,\;\alpha\succ \beta\}.\]
So here $X_1\subseteq A$ and $X_2\subseteq B$ are double-indexed by  \[\mathcal U_1=\{(\alpha,\beta)\in\Lambda\times\Lambda_1: \beta\succ\alpha\}\quad\text{
and}\quad\mathcal U_2=\{(\alpha,\beta)\in\Lambda_2\times\Lambda: \alpha\succ\beta\},\] respectively.
The sets $\mathcal U_1$
and $\mathcal U_2$ are clearly vertically cofinal and horizontally cofinal, respectively.

In the first situation, for every $\alpha\in \Lambda,$ the approximate identity property yields
 \[
\lim_{\overset{\beta\in \Lambda_1}{\beta\succ\alpha}}\norm{x_{\alpha\beta}-a_\alpha a_\beta}=
\lim_{\overset{\beta\in \Lambda_1}{\beta\succ\alpha}}\|a_\alpha-a_\alpha a_\beta\| =\lim_\beta \|a_\alpha-a_\alpha a_\beta\|=0.\]
Similarly,  for each $\beta\in \Lambda$,
 \[
\lim_{\overset{\alpha\in \Lambda_2}{\alpha\succ\beta}}\norm{x_{\alpha\beta}-a_\alpha a_\beta}=
\lim_{\overset{\alpha\in \Lambda_2}{\alpha\succ\beta}}\|a_\beta-a_\alpha a_\beta\| =\lim_\alpha \|a_\beta-a_\alpha a_\beta\|=0\]

In the second situation, the weak TI-net property yields

 \begin{align*}&
\lim_{\overset{\beta\in \Lambda_1}{\beta\succ\alpha}}\norm{x_{\alpha\beta}-a_\alpha a_\beta}=
\lim_{\overset{\beta\in \Lambda_1}{\beta\succ\alpha}}\|a_\beta-a_\alpha a_\beta\| =\lim_\beta \|a_\beta-a_\alpha a_\beta\|=0\quad
\text{for every}\quad\alpha\in \Lambda,\\&
\lim_{\overset{\alpha\in \Lambda_2}{\alpha\succ\beta}}\norm{x_{\alpha\beta}-a_\alpha a_\beta}=
\lim_{\overset{\alpha\in \Lambda_2}{\alpha\succ\beta}}\|a_\alpha-a_\alpha a_\beta\| =\lim_\alpha \|a_\alpha-a_\alpha a_\beta\|=0\quad\text{for every}\quad \beta\in \Lambda.\end{align*}

 Hence,  in each case,  $X_1$  approximates segments of $T_{AB}^{u}$ and $X_2$  approximates segments of $T^{l}_{AB}$.

 It is clear that $X_1$ is vertically injective and $X_2$ is horizontally injective.
Since by assumption, the norm of each $x_{\alpha\beta}$ is one and $X_1\cup X_2$ is an orthogonal $\ell^1(\eta)$-set,  Theorem \ref{TTr}  provides a
linear   isometry  of $\ell^\infty(\eta)$ into $\mA^*/\wap(\mA)$. In particular, $\mA$ is non-Arens regular.

  This isometry together with the condition $\eta\geq d(\mA)$  implies,  by Lemma \ref{hul}, that there is an isometry of $\mA^*$ into $\mA^*/\wap(\mA)$.  $\mA$ is therefore  isometrically enAr.
\end{proof}

\section{Extreme non-Arens regularity of the weighted convolution algebras}
We apply in this section Theorem \ref{TTr} to the weighted semigroup algebra of an infinite discrete weakly cancellative semigroup, the weighted group algebra  and the weighted measure algebra of an
infinite locally compact group. In each case, the algebra is r-enAr, where $r$ is at most equal to the diagonal bound of the weight. When $w=1$, these algebras are all isometrically enAr.
When $G$ is non-discrete, the weighted group algebra is isometrically enAr for any weight.

\begin{sdefinitions}\label{sdefs}
We first introduce  some terminology concerning weighted convolution algebras, for a more detailed discussion we refer the reader to \cite{dales-lau}.

Let $S$ be a semigroup with a topology.
 \begin{enumerate}
\item
A \emph{weight} on  $S$ is a continuous function
$w: S\to (0,\infty)$ which is submultiplicative, that is, with
\[
w(st)\le w(s)w(t)\quad\text{for every}\quad s,t\in S.
\]
If $S$ is a group with identity $e$, we shall assume in addition that $w(e)=1.$
\medskip

\item Following  \cite{baker-rejali} and \cite{craw-young}, we let $\Omega$ be the continuous function on $S\times S$
given by \[\Omega(s,t)=\frac{w(st)}{w(s)w(t)}.\] Note that $0<\Omega(s,t)\le 1$ for every $s,$
$t\in S.$
\item  The weight function is called
\textit{diagonally bounded} on $S$ if there exists $c>0$ such that
\[
w(s)w(t)\leq c\, w(st)\quad\text{whenever}\quad s,t\in S.\]
 In other words, the weight  function is diagonally bounded if $\Omega(s,t)\ge \frac1c$ for every $s$, $t\in S.$
\medskip

When $S$ is a group, it is usual to define the weight $w$ as diagonally bounded by $c>0$ when
\[
\sup_{s\in S} w(s)w(s^{-1})\le c.\]
It is easy to check that the two definitions are the same in this case.
\item  Let $G$ be a locally compact group.
 For a function space $\mathscr F(G)$ contained in $L^\infty(G)$, the corresponding weighted space is defined, following \cite{gronbaek} and \cite{dales-lau} as
 \[\mathscr F(G, w^{-1}) = \{f: S\to \C\colon  w^{-1} \, f \in \mathscr F(G),\}
\]
with the norm given by  $\|f\|_w=\|w^{-1} f\|_\infty$, for any $f\in \mathscr F(G, w^{-1})$.

If $\mathscr F(G)$ is contained in $L^1(G)$, then the weighted space for $w$ is defined as \[\mathscr F(G, w ) = \{f: S\to \C\colon  w  \, f \in \mathscr F(G)\}
\]
and the norm given by  $\|f\|_w=\|w f\|_1$, for any $f\in \mathscr F(G, w)$.

The space  $L^\infty(G,w^{-1})$ can then be identified with the Banach dual space of $L^1(G,w)$  via the pairing \[<f,\phi>=\int_G f(x)\overline{\phi(x)} dx\]
 for each $f\in L^1(G,w)\;\text{and}\; \phi\in L^\infty(G,w^{-1}).$
When $w\ge 1$, $L^1(G, w)$ is called \emph{Beurling algebra} and is studied for instance in \cite{dales-lau}.

In the same vein $M(G,w)$
will denote the space of all complex-valued measures  regular Borel measures on $G$ such that
\[\norm{\mu}_w=\int_G w(s) d|\mu|(s)\]
is finite.  $(M(G,w),\norm{\mu}_w)$ can be identified with the Banach dual space of $C_0(G,w^{-1})$  as defined above.
\item  When $S$ is a discrete semigroup, we shall consider the semigroup algebra $\ell^1(S)$,
its Banach dual space $\ell^\infty(S)$ of all bounded functions on $S$, and their corresponding weighted spaces $\ell^1(S, w)$ and $\ell^\infty(S, w^{-1})$ which are defined exactly as done above in the group case.
\end{enumerate}
\end{sdefinitions}
  \begin{remark}
	
Let $\WAP$ be the space of weakly almost periodic functions on a locally compact group (or on a discrete semigroup) $G$. The  literature contains a number of different definitions  for the weighted space $\WAPw$ of $\WAP$.
 The latest is in   \cite{dales-lau}, where the space $\WAPw$ is defined as a subspace of $L^\infty(G,w^{-1})$, as  in  \ref{sdefs} (iv) above.
With this definition,  using  that $L^\infty(G)/\WAP$ contains an isometric copy of $L^\infty(G)$ (\cite[Theorem B]{FG} for the non-discrete case and \cite[Theorem 4.3]{BF} or \cite[Theorem 3.3]{FV}, for the discrete one),
it is very quick to show that the quotient
$L^\infty(G, w^{-1})/\WAPw$ contains an isometric copy of $L^\infty(G, w^{-1})$ for any infinite locally compact group $G.$
However, this does not show that $L^\infty(G, w)$ is enAr, for  $\wap(G,w^{-1})$ can be different from $\W(L^1(G,w))$. They  can actually be very different. If $w$ is the weight on $\Z$ given by $w(n)=(1+n)^\alpha$, with $\alpha>0,$ it is shown in \cite[Example
9.1]{dales-lau}   that $\ell^1(\Z,w)$ is Arens regular, i.e., $\mathscr{WAP}(\ell^1(\Z,w))=\ell^\infty(\Z,w^{-1}).$  Since, as mentioned above, the quotient $\ell^\infty(\Z,w^{-1})/\wap(\Z,w^{-1})$ is at least as large as $\ell^\infty(\Z,w^{-1})$, we conclude that $\wap(\Z,w^{-1})$ is  much smaller that  $\mathscr{WAP}(\ell^1(\Z,w))$ in this case.

A more suitable definition for the weighted almost periodic functions in the context of Arens regularity was given by Baker and Rejali in \cite{baker-rejali}.  We will not need this definition of the space of functions $\WAPw$ in this paper (Theorem \ref{TTr} 
deals directly with the space
of functionals $\W(L^1(G, w))$)
and we will not go any further with this matter at the moment, but we hope to return to it in forthcoming work.
\end{remark}

\subsection{Weighted semigroup algebras}

We start with the weighted semigroup algebra of an infinite, discrete, weakly cancellative semigroup.
We see here that, whenever the weight is diagonally bounded by some $c>0$, Theorem \ref{TTr} applies and shows that $\ell^1(S,w)$ is $r$-enAr, where $r\le c.$
 For $w=1,$ $\ell^1(S,w)$ is therefore isometrically enAr. This latter fact was proved also in  \cite{FV}, in \cite[Theorem 4.5]{BF} and \cite[Theorem 6.4]{FG}.

Recall first that a semigroup $S$ is called \textit{weakly cancellative} if the sets
\[
s^{-1}t=\{u\in S:su=t\}\quad\text{ and }\quad ts^{-1}=\{u\in S:us=t\}
\]
are finite for every $s,t\in S$. We shall use the notations
\[
s^{-1}B=\{t\in S:st\in B\}\quad\text{ and}\quad A^{-1}B =
\bigcup_{s\in A}  s^{-1}B,
\]
where $s\in S$ and $A,B\subseteq S$. The sets $Bs^{-1}$ and $BA^{-1}$
are defined similarly.

\begin{theorem}\label{thm:quotient-non-separable}
  Let $S$ be a weakly cancellative, infinite, discrete semigroup, and let $w$ be a weight on $S$ that is diagonally bounded with bound $c$.	
Then the weighted semigroup algebra $\ell^1(S,w)$ is $r$-enAr, where $r\le c.$

In particular, $\ell^1(S,w)$ is isometrically enAr when $w=1.$
\end{theorem}

\begin{proof}
Let the weight $w$ on $S$ be diagonally bounded by $c>0$ so that
\[\frac1c\le\Omega(s,t)=\frac{w(st)}{w(s)w(t)}\le1\] for every $s$, $t\in S.$

Let $\eta=|S|,$
and let $\left\{S_\alpha\right\}_{\alpha<\eta}$ be an increasing cover of $S$ made of subsets with
$|S_\alpha|\le \alpha$ for every $\alpha<\eta,$ and collect  by induction  a faithfully indexed set $X=\{s_\alpha:\alpha<\eta\}$ such that

\begin{equation}\left(S_\alpha s_\alpha\right)\cap \left(S_\beta s_\beta\right)=\left(s_\alpha S_\alpha\right)
\cap \left(s_\beta S_\beta\right)=\emptyset \quad \mbox{ for every } \alpha<\beta<\eta. \label{thin}\end{equation}


This is possible since $S$
is weakly cancellative, and so, for every $\alpha,\beta<\eta$,
 \[|(S_\beta^{-1} S_\alpha s_\alpha)\cup (s_\alpha S_\alpha  S_\beta^{-1})|\le\max\{\alpha, \beta\}<\eta.\]

Note that for each $\alpha<\eta,$ there exists $\beta(\alpha)<\eta$, $\beta(\alpha)>\alpha$ such  that
\begin{equation}s_\alpha s_\beta\in S_\beta s_\beta\quad\text{ for every}\quad \beta\ge\beta(\alpha).\label{thin2}\end{equation}
Similarly, for each $\beta<\eta$, there exists $\alpha(\beta)<\eta$,
$\alpha(\beta)>\beta$ such that \begin{equation}s_\alpha s_\beta\in s_\alpha S_\alpha\quad\text{for every}\quad
\alpha\ge\alpha(\beta).\label{thin3}\end{equation}

Split $\Lambda$ into two cofinal subsets $\Lambda_1$ and $\Lambda_2$, let \begin{align*}
  &A = \left\{\frac{\delta_{s_\alpha}}{w(s_\alpha)}:\alpha\in \Lambda_1\right\} \quad \mbox{ and }\quad
  B = \left\{\frac{\delta_{s_\alpha}}{w(s_\alpha)}:\alpha\in \Lambda_2\right\}
\end{align*}
and note that \[\frac{\delta_{s_\alpha}}{w(s_\alpha)}\ast\frac{\delta_{s_\beta}}{w(s_\beta)}=\frac{\delta_{s_\alpha s_\beta}}{w(s_\alpha)w(s_\beta)}.\]

Now, for each $\alpha<\eta$,   define

\[x_{\alpha\beta}= \frac{\delta_{s_\alpha s_\beta}}{w(s_\alpha)w(s_\beta)},\;\text{if }\;  \beta\in \Lambda_1,\;\beta\ge \alpha(\beta),\]
and for each $\beta<\eta$, define \[x_{\alpha\beta}= \frac{\delta_{s_\alpha s_\beta}}{w(s_\alpha)w(s_\beta)},\;\text{if }\;  \alpha\in \Lambda_2,\;
\alpha\ge \beta(\alpha).\]
Then put \begin{align*}X_1&=\{x_{\alpha\beta}\colon \alpha<\eta,\, \beta\in \Lambda_1, \beta\ge \alpha(\beta) \}\quad \mbox{ and  }\\X_2&=\{x_{\alpha\beta}\colon \alpha\in \Lambda_1, \, \beta\in \Lambda_2, \alpha \ge
\beta(\alpha) \},\end{align*}
that is, $X_1$ and $X_2$ are double-indexed by the vertically and horizontally cofinal sets \[\{(\alpha,\beta)\in\Lambda\times \Lambda_1: \beta\ge \alpha(\beta)\}\quad\text{and}\quad \{(\alpha,\beta)\in\Lambda_2\times \Lambda: \alpha\ge \beta(\alpha)\},\] respectively.
Using properties (\ref{thin}), (\ref{thin2}) and (\ref{thin3}), we see that $X_1$ and $X_2$ are disjoint and are, respectively, vertically and horizontally injective (and hence have  cardinality $\eta$).
 Since  $S(\delta_s/w(s))=w(s)\mathbbm{1}_{\{s\}}\in \ell^\infty(S,w^{-1})$ for each $s\in S$, and $\frac1c\leq\norm{x_{\alpha\beta}}_w\leq 1$ for every $x_{\alpha\beta}\in X_1\cup X_2$, we see that $X_1\cup X_2$ is an orthogonal $\ell^1(\eta)$-set with bound $1$ and constant $K=1/c$ in the sense of Definition \ref{def:appl1}. Since   $X_1$   approximates  segments in $T_{AB}^{u}$ and $X_2$ approximates segments in  $T_{AB}^l $, we see that  all conditions of Theorem \ref{TTr} are met so that the desired result that $\ell^1(S, w)$ is $r$-enAr with $r\le c.$

When $w=1$, $\ell^1(S, w)$ is clearly isometrically enAr.
\end{proof}


\subsection{Weighted group algebras}

We start with the weighted analogue of \cite[Theorem 2]{ulger}. With the same proof, this result  is valid for any weight function $w$ on $G.$

\begin{lemma}\label{thm:weighted_wap_l^1}  Let $G$ be a locally compact group. Then for any weight $w$ on $G$, we have $\W(L^1(G,w))\subseteq \UCw$.
\end{lemma}

\begin{theorem} \label{Eureka} Let $G$ be an infinite, non-discrete, locally compact group and  $\mathcal F(w)$ be any closed subspace of
  $\CBw.$ Then there exists a linear isometric copy of $L^\infty(G,w^{-1})$ in the quotient space ${L^\infty(G,w^{-1})}/{\mathcal F}(w).$
	\end{theorem}

\begin{proof} By \cite[Theorem 6.3]{FG}, there is a linear isometry of
$L^\infty(G)$ into $L^\infty(G)$ which is $1$-preserved by $\CB(G)$. Since $L^\infty(G)$ and $\CB(G)$ are linearly isometric to their weighted analogues,
this gives a linear isometry of $L^\infty(G, w^{-1})$ into $L^\infty(G,w^{-1})$ which is $1$-preserved by $\CBw.$
Lemma \ref{quotient} provides then the desired linear isometry $L^\infty(G,w^{-1})$ into $L^\infty(G,w^{-1})/\mathcal F(w).$
\end{proof}

This leads immediately to the isometric enArity of $L^1(G,w)$ for any weight $w$  on $G$ when
$G$ is infinite and non-discrete.

\begin{corollary}
Let $G$ be an infinite, non-discrete, locally compact group and $w$ be any weight on $G.$ Then
 the weighted group algebra $L^1(G,w)$ is isometrically enAr for any weight function $w$ on $G.$
\end{corollary}

\begin{proof} Since Lemma \ref{thm:weighted_wap_l^1} shows that \[\W(L^1(G,w))\subseteq \UCw\subseteq \CBw,\] Theorem \ref{Eureka} provides the required isometry
\[L^\infty(G,w^{-1})\to L^\infty(G,w^{-1})/\wap(L^1(G,w)).\]
 \end{proof}

Here are our corollaries.

\begin{corollary} \label{enAr} Let $G$ be an infinite locally compact group.
\begin{enumerate}
\item If $G$ is not discrete, then the weighted group algebra $L^1(G,w)$ is isometrically enAr for any weight $w$ on $G.$
\item If $G$ is an infinite discrete group, then the  weighted group algebra $\ell^1(G,w)$ is $r$-enAr, where $r\le c,$ for any
weight $w$ on $G$ that is diagonally bounded with bound $c$.
\end{enumerate}
\end{corollary}

\begin{remark} \label{wAr} As mentioned earlier, with the weight given on $\Z$ by $w(n)=(1+n)^\alpha$, $\alpha>0$, the weighted group algebra $\ell^1(\Z,w)$ is even Arens regular.  So when $w$ is not diagonally bounded, Corollary \ref{enAr} fails badly. \end{remark}

\subsection{Weighted measure algebras} Let $G$ be an infinite locally compact group. Next theorem shows that  $M(G,w)$ is $r$-enAr with $r\le c$
when  the weight $w$ is diagonally bounded by $c>0$. In particular,  $M(G)$ is isometrically enAr.
As in Remark \ref{wAr}, the theorem fails when $w$ is not diagonally bounded.

\begin{theorem}\label{th:wmg} Let $G$ be an infinite locally compact group $G$ and $w$ be a diagonally bounded weight on $G$ with bound $c>0$. Then
the weighted measure algebra $M(G,w)$ is $r$-enAr with $r\le c.$ In particular, $M(G)$ is isometrically enAr.
\end{theorem}

\begin{proof} For $\eta=|G|$, let $A$, $B$, $X_1$ and
$X_2$ be as in the proof of Theorem \ref{thm:quotient-non-separable}.
Regarding these  as subsets of $M(G,w)$ and regarding the elements $S(\delta_s/w(s))=w(s)\mathbbm{1}_{\{s\}}$ as  projections  in $M(G,w)^*$, all  the conditions for Theorem \ref{TTr} to apply are satisfied.
Therefore, we have a linear isomorphism of
$\ell^\infty(\eta)$ in the quotient $M(G,w)^*/\wap(M(G,w))$   with distortion at most $c$.

Since by  \cite[Theorem 5.5]{HN1},
 $|G|$ is the density of $M(G)$, Corollary \ref{enAr} yields the theorem.
\end{proof}

\section{The Fourier algebra}
In this section we apply  Theorem \ref{TTr} to show  that
\begin{enumerate}
\item [(1)] for every locally compact group $G$, there is a linear  isometry  from $\ell^\infty(\chi(G))$  into $VN(G)/\wap(A(G))$, and
\item  [(2)] the existence in $G$ of an open, non-compact, amenable   subgroup implies that  there is a linear  isometry  from $\ell^\infty $  into $VN(G)/\wap(A(G))$.
\end{enumerate}
Since the density character of $A(G)$ is  $\max\{\kappa(G),\chi(G)\}$,  (1) shows  automatically that the Fourier algebra $A(G)$ is  isometrically enAr for those groups with  $\chi(G)\geq \kappa(G)$.
This implies  Hu's theorem \cite{Hu97}, mentioned in the introduction, to the effect that in this situation,
$A(G)$ is enAr in the sense of Granirer.
The interesting application of (2) is when $G$ is discrete (otherwise, (2) is an easy consequence of (1) even without assuming that $G$ has an amenable subgroup).
So when $G$ is a countable discrete group with an infinite amenable subgroup (such as the free group $\F_r$ with $r$ generators, where $r\ge 2$), Statement (2) implies that $A(G)$ is isometrically enAr.

We  summarize first the basic facts on the Fourier algebra that will be needed in the remainder of this section. For more details, the reader is directed to  \cite{eyma64} or Chapter 2 of \cite{kanilaubook}.

The Fourier algebra is the collection of all functions $h$ on $G$ of the form
$h=\overline{f}\ast \check g$, with $f,g\in L^2(G)$ and $\check g(s)=g(s^{-1})$. The norm of $A(G)$ is given by \[\|h\|=\inf\{\|f\|_2\|g\|_2:h=\overline{f}\ast \check g,\; f,\,g\in L^2(G)\}.\]

We may remark that, when $G$ is abelian, $A(G)$ identifies with $L^1(\widehat G)$ via the Fourier transform. The results stated in Section 4 for the group algebra show therefore that the Fourier algebra $A(G)$ is isometrically enAr when $G$ is Abelian. Non-Arens regularity of $A(G)$ has however turned out to be more resistant and a complete solution to the regularity problem for the Fourier algebra is not known yet.

The Banach dual of $A(G)$ is isometrically isomorphic to the group von Neumann algebra $VN(G)$, which is the closure in the weak operator topology of the linear span of $\{\lambda(x) : x \in G\}$
in $\mathcal B(L^2(G))$, where $\lambda$ is the left regular representation of $G$ on $L^2(G)$.
This linear isometry identifies each $T\in VN(G)$ with an element $\varphi_T\in A(G)^\ast$ such that
\[\varphi_T\left(\overline f\ast \check g\right)=\<{Tg,f}>
,\] where the bracket referst to the $L^2(G)$ inner product.

Under this identification, normal sates of $VN(G)$ correspond to the set
 \[\mathcal{P}_1(G)=\{\varphi\in A(G)\colon \varphi\; \text{is positive definite and}\; \|\varphi \|= \varphi(e) = 1\}.\]
So here, a TI-net is a net $\{\varphi_\alpha\colon \alpha \in \Lambda\}$  in $\mathcal{P}_1(G)$
 with the property \[\lim_{\alpha} \left\| \varphi _\alpha \varphi-\varphi_\alpha\right\|=\lim_{\alpha} \left\| \varphi\varphi _\alpha -\varphi_\alpha\right\|=0\quad\text{ for every}\quad\varphi\in \mathcal{P}_1(G).\]

		Our approach is based on the following two lemmas:
\begin{lemma}[Theorem 2.4 of \cite{chou82TIM}] \label{chou2.4}
Let $\{a_n\}_{n\in \N}$  be a sequence of normal states of a von Neumann algebra $\mV$ such that $\lim_m \norm{ a_n-a}=2$ for each normal state of $\mV$.
$a\in \mV_\ast$.
Then there exist positive integers $n_1<n_2<\cdots$ and a sequence of  normal states
$\{b_j\}_{j\in \N}$ such that
\begin{enumerate}
  \item $\lim_j \norm{a_{n_j}-b_j}=0$ and
  \item The sequence $\{b_j\}_{j\in \N}$  is an orthogonal $\ell^1$-set.
\end{enumerate}
\end{lemma}

\begin{lemma}
\label{VN}Let $G$ be a locally compact group and let  $Q$  be  a  projection in $VN(G)$. Take $h\in L^2(G)$, with $Qh\neq 0$ and define  $\phi=\frac{1}{\|Qh\|_2}Q h\in L^2(G)$ and $\psi=\phi \ast \widetilde{\phi}$. Then  $\psi\in A(G)\cap \mathcal{P}_1(G)$ and $S(\psi)\leq Q$.
\end{lemma}

\begin{proof}
It is clear that $\psi\in A(G)\cap \mathcal{P}_1(G)$. To prove that $S(\psi)\leq Q$, we only have to recall that $S(\psi)$ is the smallest projection in $VN(G)$ with $\langle \psi, S(\psi)\rangle=\norm{\psi}$ and observe that:
\begin{align*}
  \left\langle \psi,Q \right\rangle_{_{\langle A(G),VN(G)\rangle}}&=
  \left\langle Q\phi,\phi \right\rangle_{_{\langle L^2(G),L^2(G)\rangle}}\\&=\frac{1}{\norm{Qh}_2^2}
  \left\langle Q Q h,Qh \right\rangle_{_{\langle L^2(G),L^2(G)\rangle}}\\&=1=\psi(e)=\|\psi\|.
\end{align*}
\end{proof}

 Next we proceed to check that the Fourier algebra of a non-discrete locally compact group contains TI-nets which are orthogonal $\ell^1(\chi(G))$-sets.

\begin{theorem}\label{Gentent}
If $G$ is a non-discrete locally compact group, then $A(G)$ contains a TI-net of true cardinality $\chi(G)$ that is an orthogonal $\ell^1(\chi(G))$-set.
\end{theorem}
\begin{proof}
If $G$ is metrizable, this was shown by Chou in \cite{chou82TIM}. The proof there consists in observing that every TI-sequence in $\mathcal P_1(G)$ satisfies the conditions of Lemma \ref{chou2.4} (this is \cite[Lemma 3.2]{chou82TIM}) and, hence, has a subsequence  that can be approximated by  another $TI$-sequence  which  is also an orthogonal $\ell^1$-set.   TI-sequences that are orthogonal $\ell^1$-sets can therefore be found as long as
    TI-sequences are  available, which is the case in non-discrete groups, see \cite[Proposition 3]{rena72} or  \cite[Lemma 5.2]{FGnew}.

We now  assume that $G$ is not metrizable. Let $\eta=\chi(G)$ and
 $\{U_\alpha\colon \alpha<\eta\}$ be a base of symmetric neighbourhoods of the identity $e$. For each $\alpha<\eta$, let $V_\alpha$ be a neighbourhood of $e$ with $V_\alpha^4\subset U_\alpha$ and put
 $\mathcal{B}=\{V_\alpha\colon \alpha<\eta\}$.

We consider the family of compact subgroups $\{N_\alpha\colon \alpha<\eta\}$ given by the Kakutani-Kodaira theorem for the base $\mathcal B$ (see \cite[Proposition 4.3 and its proof]{Hu95}).
Recall,  in particular, that $N_{\alpha+1}\subset N_\alpha \cap
V_\alpha$ for every $\mathbf{\alpha<\eta}$.
Consider then the projections $P_\alpha \colon L^2(G) \to L^2(G/N_{\alpha+1})$ given by $P_\alpha(f)= \lambda_{{N_\alpha+1}}\ast f$, where $L^2(G/N_{\alpha+1})$ stands for the functions of $L^2(G)$ that are constant on  the cosets of  $N_{\alpha+1}$ and $\lambda_{{N_\alpha+1}}$ stands for the regular representation of $N_{\alpha+1}$. This is an increasing net of projections. Putting $Q_\alpha=P_{\alpha+1}-P_\alpha$, we obtain  an orthogonal net of projections $\{Q_\alpha\colon \alpha<\eta\}$.

Since $N_{\alpha+1}\setminus N_{\alpha+2}\ne\emptyset$ and $N_{\alpha}$'s are closed, we may pick for each $\alpha<\eta$,   a symmetric compact neighbourhood $W_\alpha\subset V_\alpha$ of $e$
such that  $N_{\alpha+1}\setminus N_{\alpha+2}W_\alpha\ne\emptyset$, and so $N_{\alpha+1} W_\alpha\setminus N_{\alpha+2}W_\alpha$
has non-empty interior in $G$.
Define
$h_\alpha:=1_{N_{\alpha+2}W_{\alpha}}$.  Put
\[\phi_\alpha=\frac{1}{\|Q_\alpha
h_\alpha\|_2}Q_\alpha h_\alpha\quad\text{ and}\quad
\psi_\alpha=\phi_\alpha\ast
\widetilde{\phi_\alpha}.\]
Note that $P_{\alpha+1}h_\alpha=h_\alpha,$ and so $Q_\alpha h_\alpha=h_\alpha-P_\alpha h_\alpha$ for each $\alpha<\eta$.
We claim that $h_\alpha\ne P_\alpha h_\alpha$ so that $Q_\alpha h_\alpha$ is not zero for each $\alpha<\eta$.
Let $x=pw$ be any point in $N_{\alpha+1} W_\alpha\setminus N_{\alpha+2}W_\alpha$ with $p\in N_{\alpha+1} $ and $w\in W_\alpha$.
Then $h_\alpha(x)=0$, while
 \begin{align*}
P_\alpha h_\alpha (x)&=P_\alpha h_\alpha (w)\\
  &=\int_{N_{\alpha+1}} h_\alpha (t^{-1}w)d\lambda_{N_{\alpha+1}}(t)=
  \lambda_{{N_\alpha+1}}\left(wW_\alpha N_{\alpha+2}\cap N_{\alpha+1}\right)\neq 0,
\end{align*} where
the latter value is non-zero because the interior in $N_{\alpha+1}$ of the set $wW_\alpha N_{\alpha+2}\cap N_{\alpha+1}$  is non-empty since it   contains  $N_{\alpha+2}$.

Hence $P_\alpha h_\alpha $ and  $h_\alpha$ differ on the set $N_{\alpha +1}W_\alpha\setminus N_{\alpha+2}W_\alpha$, which is of positive measure in $G$ (having non-empty interior).
We conclude that $P_\alpha h_\alpha \neq h_\alpha $ for each $\alpha<\eta$, as wanted.

 Lemma \ref{VN} then implies that $S(\psi_\alpha)\leq Q_{\alpha}$,  showing that  $\{\psi_\alpha \colon \alpha<\eta\}$ is an orthogonal $\ell^1(\eta)$-set.

Since the support of $Q_{\alpha} h_{\alpha}= h_{\alpha}-\lambda_{N_{\alpha+1}}\ast h_{\alpha}$ is clearly contained in $N_{\alpha+1} W_{\alpha}$, which is in turn contained in
$V_\alpha^2$,  we
see that \[\supp(\psi_\alpha)=\supp\left(\phi_\alpha\ast \widetilde{\phi_\alpha}\right)\subseteq V_\alpha^4\subseteq U_\alpha\quad\text{
for each}\quad\alpha<\eta,\]  and  \cite[Proposition 3]{rena72} proves that  $\{\psi_\alpha \colon \alpha<\eta\}$ is  a TI-net as well. Since this net is directed with the natural order of the ordinal $\eta$, its true cardinality is $\eta$.
\end{proof}

When $G$ is discrete, Theorem \ref{Gentent} is useless, but, in case $G$ is amenable, we can replace TI-nets by weak bounded approximate identities.

Following \cite{eyma64}, let $P(G)$ be the space of continuous positive definite
functions on $G$ and $B(G)$ be its the linear span.
The space $B(G)$ is a Banach algebra,  called the the Fourier-Stieltjes algebra, and if $C^*(G)$ is the group
$C^*\text{-algebra}$ of $G$, then $B(G)$ is its Banach dual.

\begin{theorem}
  \label{countamen}
  If $G$ is a locally compact group that contains a $\sigma$-compact, non-compact  open amenable subgroup $H$,   then $A(G)$ has an orthogonal weak bai.
  \end{theorem}
  \begin{proof}
    It is a well-known theorem of Leptin that $A(H)$ contains a sequential  bai $\{v_n\}_{n\in \N}$, see, e.g., \cite[Theorem 2.7.2]{kanilaubook}. It is then clear that, in the $\sigma(B(H),C^\ast(H))$-topology, $\lim_n v_n=\mathbf{1}$ where $\mathbf{1}$ denotes the constant 1-function.

Since $H$ is  not compact, the regular representation  of $H$ is disjoint from the trivial one-dimensional representation. It then   follows from \cite[Corollaire 3.13]{arsac76} (or \cite[Proposition 2.8.9]{kanilaubook}) that, for any $u\in A(H)$,
\begin{equation*}\norm{\mathbf{1}-u}_{_{B(H)}} =1+\norm{u}_{_{A(H)}} .\end{equation*}

Let now $u\in A(H)$ be an arbitrary   positive definite function with $\norm{u}_{A(H)}=1$ (i.e., an arbitrary   normal state $u$ of $VN(H)=A(H)^\ast$). Given $\varepsilon>0$ there is then $T_\varepsilon\in C^\ast(H)$ with $\norm{T_\varepsilon}\leq 1$ such that
\[\left|
\<{
\mathbf{1}-u, T_{\varepsilon}
}
>\right| > 2-\varepsilon.\]
As a consequence, there is $n_\varepsilon \in \N$ such that, for $n\geq n_\varepsilon$,
\[ \left|\<{v_n-u, T_{\varepsilon}} >\right|\geq 2-\varepsilon.\]
It follows that
\begin{equation}\label{disjoint}
\lim_n \norm{v_n-u}_{_{A(H)}}=2.
\end{equation}
 Lemma \ref{chou2.4} now  provides an orthogonal $\ell^1$-sequence $\{u_j \}_{j\in \N}$ and a subsequence $\{v_{n_j}\}_{j\in \N}$   of $\{v_{n}\}_{n\in \N}$ such that
\begin{equation}\label{app}\lim_j \norm{v_{n_j}-u_j}_{_{A(H)}}=0.\end{equation}

We next consider the    restriction and extension maps, $R\colon A(G)\to A(H)$ and  $\Phi\colon A(H)\to A(G)$, the latter one defined by
$\Phi(u)(s)=u(s)$ if $s \in H$ and $\Phi(u)(s)=0$ if $s\notin H$. The adjoint $R^\ast\colon VN(H)\to VN(G)$ of $R$ is then   a multiplicative linear isometry (see \cite[Proposition 7.3.5]{deri11}, this is considerably easier when $H$ is open) and $\norm{\Phi(u)}_{A(G)}=\norm{u}_{A(H)}$, see \cite[Proposition 2.4.1]{kanilaubook}.

The sought after orthogonal weak bai will be the sequence $(\Phi(u_j) )_{j\in \N}$ as we check next.

Let $\{S(u_j)\}_{j\in\N}$ be the family of orthogonal projections corresponding to sequence $\{u_j \}_{j\in \N}$, and consider, for each $j\in \N,$ the operator in $VN(G)$ given by $R^*(S(u_j))$.

 Since normal states of $VN(G)$ correspond precisely  to  positive definite functions of $A(G)$ and these are clearly preserved by  $R$, $R^\ast$ must preserve self-adjointness. This, together with the multiplicative character of $R^\ast$, implies that the operators
$R^\ast(S(u_j))$ are  projections.
Finally, since \[\langle \Phi(u_j),R^*(S(u_j))\rangle=\langle u_j,S(u_j)\rangle=1=u_j(e)=\Phi(u_j)(e)\] for each $j\in\N$, we deduce according to Definition \ref{def:supp},  that $S(\Phi(u_j))\le R^*(S(u_j))$ for each $j\in \N.$
Thus $(\Phi(u_j) )_{j\in \N}$ is orthogonal in the sense of Definition \ref{def:appl1}.

That  $(\Phi(u_j) )_{j\in \N}$ is a weak bai follows from the bai property of the sequence $(v_{n_j})_{j\in \N}$, the approximation property  \eqref{app} and the following inequality, valid for every $j,k\in \N$:
\begin{align*}
  \norm{\Phi(u_j)\Phi(u_k)-\Phi(u_k)}_{_{A(G)}}&=\norm{u_ju_k-u_k}_{_{A(H)}}\\
  &\leq \norm{u_ju_k-v_{n_j}u_k}_{_{A(H)}}+ \norm{v_{n_j}u_k-u_k}_{_{A(H)}}. \end{align*}
\end{proof}

\begin{corollary}\label{cor:ISOVN}
 Let $G$ be a locally compact group.  $A(G)$ is isometrically enAr if  $G$ satisfies any of the following conditions:
 \begin{enumerate}
   \item  $\chi(G)\geq \kappa(G)$, or
   \item $G$ is second countable and  contains a non-compact open   amenable subgroup.
  \end{enumerate}
\end{corollary}
\begin{proof} (1) To obtain the linear isometry
 of $VN(G)$ into $VN(G)/\wap(A(G))$, simply combine Theorems \ref{TI} and \ref{Gentent} and the fact that, under the hypothesis of (1), $d(A(G))=\chi(G)$.

 (2) We only have to put together Theorems \ref{countamen} and  \ref{TI}.
\end{proof}

\bibliographystyle{plain}

\end{document}